\documentclass[11pt]{article}
\usepackage{amssymb}
\usepackage{amsmath}
\usepackage{amsfonts}

\newtheorem{theorem}{Theorem}

\newtheorem{corollary}[theorem]{Corollary}

\newtheorem{lemma}[theorem]{Lemma}

\newtheorem{proposition}[theorem]{Proposition}
\newtheorem{remark}[theorem]{Remark}

\newenvironment{proof}[1][Proof]{\noindent\textbf{#1.} }{\ \rule{0.5em}{0.5em}}

\begin{document}

\title{Geometric Stable processes and related fractional differential
equations}
\author{Luisa Beghin\thanks{%
Address: Department of Statistical Sciences, Sapienza University of Rome,
P.le A. Moro 5, I-00185 Roma, Italy. e-mail: \texttt{luisa.beghin@uniroma1.it%
}}}
\date{}
\maketitle

\begin{abstract}
We are interested in the differential equations satisfied by the density of
the Geometric Stable processes $\mathcal{G}_{\alpha }^{\beta }=\left\{
\mathcal{G}_{\alpha }^{\beta }(t);t\geq 0\right\} $, with stability \ index $%
\alpha \in (0,2]$ and asymmetry parameter $\beta \in \lbrack -1,1]$, both in
the univariate and in the multivariate cases. We resort to their
representation as compositions of stable processes with an independent Gamma
subordinator. As a preliminary result, we prove that the latter is governed
by a differential equation expressed by means of the shift operator. As a
consequence, we obtain the space-fractional equation satisfied by the
density of $\mathcal{G}_{\alpha }^{\beta }.$ For some particular values of $%
\alpha $ and $\beta ,$ we get some interesting results linked to well-known
processes, such as the Variance Gamma process and the first passage time of
the Brownian motion.\newline
\end{abstract}

\section{Introduction and notation}

The Geometric Stable (hereafter GS) random variable (r.v.) is usually
defined through its characteristic function: let $\mathcal{G}_{\alpha
}^{\beta }$ be a GS r.v. with stability index $\alpha \in (0,2]$, symmetry
parameter $\beta \in \lbrack -1,1]$, position parameter $\mu \in \mathbb{R}$%
, scale parameter $\sigma >0$, then%
\begin{equation}
\mathbb{E}e^{i\theta \mathcal{G}_{\alpha }^{\beta }}=\frac{1}{1+\sigma
^{\alpha }|\theta |^{\alpha }\omega _{\alpha ,\beta }(\theta )-i\mu \theta }%
,\quad \theta \in \mathbb{R},  \label{koz}
\end{equation}%
where%
\begin{equation*}
\omega _{\alpha ,\beta }(\theta ):=\left\{
\begin{array}{c}
1-i\beta sign(\theta )\tan (\pi \alpha /2),\quad \text{if }\alpha \neq 1 \\
1+2i\beta sign(\theta )\log |\theta |/\pi ,\quad \text{if }\alpha =1%
\end{array}%
\right. ,
\end{equation*}%
(see e.g. \cite{JAY}). Moreover the following relationship holds (see \cite%
{KOZ3})%
\begin{equation*}
\mathbb{E}e^{i\theta \mathcal{G}_{\alpha }^{\beta }}=\frac{1}{1-\log \Phi _{%
\mathcal{S}_{\alpha }^{\beta }(1)}(\theta )}
\end{equation*}%
where
\begin{equation}
\Phi _{\mathcal{S}_{\alpha }^{\beta }}(\theta ):=\mathbb{E}e^{i\theta
\mathcal{S}_{\alpha }^{\beta }}=\exp \{i\theta \mu -\sigma ^{\alpha }|\theta
|^{\alpha }\omega _{\alpha ,\beta }(\theta )\},\qquad \theta \in \mathbb{R},
\label{can}
\end{equation}%
is the characteristic function of a stable r.v. $\mathcal{S}_{\alpha
}^{\beta }$ with the same parameters $\alpha $ $,$ $\beta ,$ $\mu ,$ $\sigma
$. We will consider, for simplicity, the case $\mu =0$; then we will refer
only to strictly stable r.v.'s, if $\alpha \neq 1.$

The main features of the GS laws are the heavy tails and the unboundedness
at zero. These two characteristics, together with their stability properties
(with respect to geometric summation) and domains of attraction, make them
attractive in modelling financial data, as shown, for example, in \cite{KOZ2}%
. As particular cases, when the symmetry parameter $\beta $ is equal to $1,$
the support of the GS r.v. is limited to $\mathbb{R}^{+}$ and its law
coincides, for $0<\alpha \leq 1,$ with the Mittag-Leffler distribution, as
shown in \cite{JAY} and \cite{KOZ3}. Moreover the GS distribution is
sometimes referred to as "asymmetric Linnik distribution", since it can be
considered a generalization of the latter (to which it reduces for $\beta
=\mu =0$, see \cite{KOZ4}, \cite{ERD}). The Linnik distribution exhibits fat
tails, finite mean for $1<\alpha \leq 2$ and also finite variance only for $%
\alpha =2$ (when it takes the name of Laplace distribution) and is applied
in particular to model temporal changes in stock prices (see \cite{AND2}).

The univariate GS process will be denoted as $\left\{ \mathcal{G}_{\alpha
}^{\beta }(t),t\geq 0\right\} $ and defined by having the one-dimensional
distribution coinciding with $\mathcal{G}_{\alpha }^{\beta }$ and
characteristic exponent equal to%
\begin{equation*}
\psi _{\mathcal{G}_{\alpha }^{\beta }}(\theta )=\log (1+\sigma ^{\alpha
}|\theta |^{\alpha }\omega _{\alpha ,\beta }(\theta )),\qquad \theta \in
\mathbb{R},
\end{equation*}%
(see \cite{SIK}, \cite{BOG}). Moreover the following representation holds%
\begin{equation}
\mathcal{G}_{\alpha }^{\beta }(t):=\mathcal{S}_{\alpha }^{\beta }(\Gamma
(t)),\qquad t\geq 0,  \label{sub}
\end{equation}%
where $\mathcal{S}_{\alpha }^{\beta }(t)$ is, for any $t,$ a stable law with
parameters $\mu =0,$ $\beta \in \lbrack -1,1],$ $\sigma =t^{1/\alpha }$ and $%
\left\{ \Gamma (t),t\geq 0\right\} $ is an independent Gamma subordinator.
We will use the following notation, for a generic process $X:=\{X(t),t\geq
0\}.$

We note that, for $\beta =0,$ the process $\mathcal{G}_{\alpha }^{\beta }$
reduces to a symmetric GS process (that we will denote simply as $\mathcal{G}%
_{\alpha }$), while, for $\beta =1,$ it is called GS subordinator (since it
is increasing and L\'{e}vy); we will denote it as $\mathcal{G}_{\alpha
}^{\prime }.$

The space-fractional differential equation that we obtain here, as governing
equations of $\mathcal{G}_{\alpha }^{\beta },$ are expressed in terms of
Riesz and Riesz-Feller derivatives. We recall that the Riesz fractional
derivative $^{R}\mathcal{D}_{x}^{\alpha }$ is defined through its Fourier
transform, which reads, for $\alpha >0$ and for an infinitely differentiable
function $u$,%
\begin{equation}
\mathcal{F}\left\{ ^{R}\mathcal{D}_{x}^{\alpha }u(x);\theta \right\}
=-|\theta |^{\alpha }\mathcal{F}\left\{ u(x);\theta \right\} ,  \label{rs2}
\end{equation}%
where the Fourier transform is defined as $\mathcal{F}\left\{ u(x);\theta
\right\} :=\int_{-\infty }^{+\infty }e^{i\theta x}u(x)dx$ (see \cite{MAI}
and \cite{KILB}, p.131). Alternatively it can be explicitly represented as
follows, for $\alpha \in (0,2]$,%
\begin{equation}
^{R}\mathcal{D}_{x}^{\alpha }u(x):=-\frac{1}{2\cos (\alpha \pi /2)}\frac{1}{%
\Gamma (1-\alpha )}\frac{d}{dx}\int_{-\infty }^{+\infty }\frac{u(z)}{%
|x-z|^{\alpha }}dz  \label{rs}
\end{equation}%
(see \cite{SAI}). The more general Riesz-Feller definition is given by%
\begin{equation}
\mathcal{F}\left\{ ^{RF}\mathcal{D}_{x,\beta }^{\alpha }u(x);\theta \right\}
=\psi _{\beta }^{\alpha }(\theta )\mathcal{F}\left\{ u(x);\theta \right\}
,\quad \alpha \in (0,2],  \label{fel}
\end{equation}%
where
\begin{equation}
\psi _{\beta }^{\alpha }(\theta ):=-|\theta |^{\alpha }e^{i\frac{\gamma \pi
}{2}sign\theta }\text{, }|\gamma |\leq \min \{\alpha ,2-\alpha \}  \label{ai}
\end{equation}%
(see \cite{KILB}, p.359 and \cite{MAI}) and $\psi _{\beta }^{\alpha }(\theta
)$ coincides with the characteristic exponent of the stable random variable $%
\mathcal{S}_{\alpha }^{\beta }$, in the Feller parametrization, for $\gamma =%
\frac{2}{\pi }\arctan \left[ -\beta \tan \frac{\pi \alpha }{2}\right] .$
Indeed (\ref{can}) can be rewritten (for $\mu =0$) as%
\begin{equation}
\Phi _{\mathcal{S}_{\alpha }^{\beta }}(\theta )=\exp \{c\psi _{\beta
}^{\alpha }(\theta )\},\qquad \theta \in \mathbb{R},\text{ }c=\sigma
^{\alpha }\left[ \cos (\pi \gamma /2)\right] ^{-1}.  \label{can2}
\end{equation}

We recall now the following result on stable processes proved in \cite{MAI}
(in the special case $c=1$), which will be used later: let $p_{\alpha
}^{\beta }(x,t),$ $x\in \mathbb{R},t\geq 0,$\ be the transition density of
the stable process $\mathcal{S}_{\alpha }^{\beta }$, then $p_{\alpha
}^{\beta }$ satisfies the following space-fractional differential equation,
for $\alpha \in (0,2],$ $x\in \mathbb{R}$, $t\geq 0$:%
\begin{equation}
\left\{
\begin{array}{l}
^{RF}\mathcal{D}_{x,\beta }^{\alpha }\,p_{\alpha }^{\beta }(x,t)=\frac{1}{c}%
\frac{\partial }{\partial t}p_{\alpha }^{\beta }(x,t) \\
p_{\alpha }^{\beta }(x,0)=\delta (x) \\
\lim_{|x|\rightarrow \infty }p_{\alpha }^{\beta }(x,t)=0%
\end{array}%
\right. ,  \label{rs5}
\end{equation}%
and the additional condition $\left. \frac{\partial }{\partial t}p_{\alpha
}^{\beta }(x,t)\right\vert _{t=0}=0$, if $\alpha >1.$

\

Our main result concerns the space-fractional equation satisfied by the
density $g_{\alpha }^{\beta }(x,t),$ $x\in \mathbb{R},t\geq 0,$ of the GS
process $\mathcal{G}_{\alpha }^{\beta }$. As a preliminary step we derive
the partial differential equation satisfied by the density $f_{\Gamma
}(x,t),x,t\geq 0,$ of the Gamma subordinator $\Gamma $ and then we resort to
the representation (\ref{sub}) of the GS process. Indeed we prove that $%
f_{\Gamma }(x,t)$ satisfies%
\begin{equation}
\frac{\partial }{\partial x}f_{\Gamma }=-b(1-e^{-\partial _{t}})f_{\Gamma
},\qquad x,t\geq 0,
\end{equation}%
where $b$ is the rate parameter of $\Gamma $ (see (\ref{gam}) below) and $%
e^{-\partial _{t}}$ is a particular case (for $k=1$) of the shift operator,
defined as%
\begin{equation}
e^{-k\partial _{t}}f(t):=\sum_{n=0}^{\infty }\frac{(-k\partial _{t})^{n}}{n!}%
f(t)=f(t-k),\quad k\in \mathbb{R},  \label{shi}
\end{equation}%
for any analytical function $f:\mathbb{R}\rightarrow \mathbb{R}$. As a
consequence, we show that $g_{\alpha }^{\beta }(x,t)$ satisfies, for $x\in
\mathbb{R}$, $t\geq 0$, $\alpha \in (0,2],$ the following Cauchy problem%
\begin{equation}
\left\{
\begin{array}{l}
^{RF}\mathcal{D}_{x,\beta }^{\alpha }\,g_{\alpha }^{\beta }(x,t)=\frac{1}{c}%
(1-e^{-\partial _{t}})g_{\alpha }^{\beta }(x,t) \\
g_{\alpha }^{\beta }(x,0)=\delta (x) \\
\lim_{|x|\rightarrow \infty }g_{\alpha }^{\beta }(x,t)=0%
\end{array}%
\right. .  \label{rs6}
\end{equation}%
In the $n$-dimensional case, we prove that the governing equation of the GS
vector process in $\mathbb{R}^{n}$ is analogous to (\ref{rs6}), but the
Riesz-Feller fractional derivative is substituted, in this case, by the
fractional derivative operator $\nabla _{M}^{\alpha }$ defined by%
\begin{equation}
\mathcal{F}\left\{ \nabla _{M}^{\alpha }u(\mathbf{x});\mathbf{\theta }%
\right\} =-\left[ \int_{S^{n}}(-i<\mathbf{z},\mathbf{\theta }>)^{\alpha }M(d%
\mathbf{z})\right] \mathcal{F}\left\{ u(\mathbf{x});\mathbf{\theta }\right\}
,\quad \mathbf{\theta ,x}\in \mathbb{R}^{n},\alpha \in (0,2],\text{ }\alpha
\neq 1,  \label{mgs}
\end{equation}%
where $S^{n}:=\{\mathbf{s}\in \mathbb{R}^{n}:||\mathbf{s}||=1\}$ and $M$ is
the spectral measure (see \cite{MEE}, with a change of sign due to the
different definition of Fourier transform). The multivariate GS law has been
first introduced in \cite{AND} (in the isotropic case) and called
multivariate Linnik distribution.

As special cases of the previous results the governing equations of some
well-known processes are obtained: indeed, in the symmetric case and for $%
\alpha =2$, the GS process reduces to the Variance Gamma process, while, for
$\alpha =1$, it coincides with a Cauchy process subordinated to a Gamma
subordinator. On the other hand, in the positively asymmetric case, $%
\mathcal{G}_{\alpha }^{\beta }$ reduces to a GS subordinator, which is used
in particular as random time argument of the subordinated Brownian motion,
via successive iterations (see \cite{BOG}, \cite{SIK}) Moreover, for $\alpha
=1/2$, we can obtain, as a corollary, the fractional equation satisfied by
the density $g_{1/2}^{\prime }(x,t)$ of the first-passage time of a standard
Brownian motion $B$ through a Gamma distributed random barrier, i.e.%
\begin{equation*}
g_{1/2}^{\prime }(x,t):=P\left\{ \inf_{s>0}\left\{ B(s)=\Gamma (t)\right\}
\in dx\right\} ,\qquad x,t\geq 0.
\end{equation*}%
Indeed we prove that $g_{1/2}^{\prime }(x,t)$ the space-fractional equation
\begin{equation}
\frac{\partial ^{1/2}}{\partial |x|^{1/2}}g_{1/2}^{\prime }(x,t)=\frac{1}{%
\sqrt{2}}(1-e^{-\partial _{t}})g_{1/2}^{\prime }(x,t),\qquad x,t\geq 0,
\end{equation}%
where $\partial ^{1/2}/\partial |x|^{1/2}:=\,^{RF}\mathcal{D}_{x,1}^{1/2},$
with the conditions in (\ref{rs6}).

\section{Preliminary results}

We start by deriving the differential equation satisfied by the density of
the Gamma subordinator, since it will be applied in the study of the
equation governing the GS process (thanks to the representation (\ref{sub})).

The one-dimensional distribution of the Gamma subordinator $\left\{ \Gamma
_{a,b}(t),\text{ }t\geq 0\right\} ,$ of parameters $a,b>0$ is given by%
\begin{equation}
f_{\Gamma _{a,b}}(x,t):=\Pr \left\{ \Gamma _{a,b}(t)\in dx\right\} =\left\{
\begin{array}{l}
\frac{b^{at}}{\Gamma (at)}x^{at-1}e^{-bx},\qquad x\geq 0 \\
0,\qquad x<0%
\end{array}%
\right. ,\quad t\geq 0.  \label{gam}
\end{equation}%
(see, for example, \cite{APPL}, p.52). Hereafter we will consider, for the
sake of simplicity, the case $a=1$ and denote $\Gamma _{1,b}:=\Gamma .$ The
Fourier transform of (\ref{gam}) is given by%
\begin{equation}
\widehat{f}_{\Gamma }(\theta ,t):=\mathcal{F}\left\{ f_{\Gamma }(x,t);\theta
\right\} =\mathbb{E}e^{i\theta \Gamma (t)}=\left( 1-\frac{i\theta }{b}%
\right) ^{-t},\qquad \theta \in \mathbb{R}.  \label{car}
\end{equation}

\begin{lemma}
The density (\ref{gam}) of the Gamma subordinator satisfies (for $a=1),$ the
following equation%
\begin{equation}
\frac{\partial }{\partial x}f_{\Gamma }=-b(1-e^{-\partial _{t}})f_{\Gamma
},\qquad x,t\geq 0,  \label{gam2}
\end{equation}%
where $e^{-\partial _{t}}$ is the partial derivative version of the shift
operator defined in (\ref{shi}), for $k=1.$ The initial and boundary
conditions are the following%
\begin{equation}
\left\{
\begin{array}{l}
f_{\Gamma }(x,0)=\delta (x) \\
\lim_{|x|\rightarrow +\infty }f_{\Gamma }(x,t)=0,\text{ }t\geq 0%
\end{array}%
\right. .  \label{gam3}
\end{equation}

\begin{proof}
The first condition in (\ref{gam3}) can be checked easily by considering (%
\ref{car}) and the definition of the Dirac delta function, i.e. $\delta (x):=%
\frac{1}{2\pi }\int_{\mathbb{R}}e^{-i\theta x}d\theta .$ The second one is
immediately satisfied by (\ref{gam}). As far as equation (\ref{gam2}) is
concerned, the Fourier transform of its left-hand side, with respect to $x,$
is given by%
\begin{eqnarray}
&&\mathcal{F}\left\{ \frac{\partial }{\partial x}f_{\Gamma }(x,t);\theta
\right\}  \label{un} \\
&=&\text{[by (\ref{gam3})]}=-i\theta \widehat{f}_{\Gamma }(\theta
,t)=-i\theta \left( \frac{b}{b-i\theta }\right) ^{t}.  \notag
\end{eqnarray}%
For the right-hand side of (\ref{gam2}) we have that%
\begin{eqnarray*}
-b\widehat{f}_{\Gamma }(\theta ,t)+be^{-\partial _{t}}\widehat{f}_{\Gamma
}(\theta ,t) &=&-b\left( \frac{b}{b-i\theta }\right) ^{t}+be^{-\partial
_{t}}\left( \frac{b}{b-i\theta }\right) ^{t} \\
&=&-b\left( \frac{b}{b-i\theta }\right) ^{t}+b\left( \frac{b}{b-i\theta }%
\right) ^{t-1},
\end{eqnarray*}%
which coincides with (\ref{un}).
\end{proof}
\end{lemma}

An alternative result on the differential equation satisfied by $f_{\Gamma }$
can be obtained by considering the following differential operator: for any
given infinitely differentiable function $f(x)$,%
\begin{equation}
\mathcal{A}_{k,x}f(x):=\sum_{j=1}^{\infty }\frac{(-1/k)^{j+1}}{j}%
D_{x}^{j}f(x),\quad x\geq 0,\text{ }k\in \mathbb{R}.  \label{tem2}
\end{equation}%
We could use for (\ref{tem2}) the formalism $\mathcal{A}_{k,x}f(x)=\log
(1+D_{x}/k).$

If moreover $\left. D_{x}^{j}f(x)\right\vert _{|x|=\infty }=0,$ for any $%
j\geq 0$, the Fourier transform of (\ref{tem2}) can be written as follows:%
\begin{eqnarray}
\mathcal{F}\left\{ \mathcal{A}_{k,x}f(x);\theta \right\}
&=&\sum_{l=1}^{\infty }\frac{(-1/k)^{l+1}}{l}\int_{-\infty }^{+\infty
}e^{i\theta x}D_{x}^{l}f(x)dx  \label{re7} \\
&=&\sum_{l=1}^{\infty }\frac{(-1/k)^{l+1}}{l}(-i\theta )^{l}\widehat{f}%
(\theta )  \notag \\
&=&\log \left( 1-\frac{i\theta }{k}\right) \widehat{f}(\theta ).  \notag
\end{eqnarray}

\begin{lemma}
The following differential equation is satisfied by the density of the Gamma
subordinator:%
\begin{equation}
\frac{\partial }{\partial t}f_{\Gamma }=-\mathcal{A}_{b,x}f_{\Gamma },\qquad
x,t\geq 0,  \label{re2}
\end{equation}%
with the conditions%
\begin{equation}
\left\{
\begin{array}{l}
f_{\Gamma }(x,0)=\delta (x) \\
\lim_{|x|\rightarrow \infty }D_{x}^{l}f_{\Gamma }(x,t)=0,\qquad l=0,1,...%
\end{array}%
\right.  \label{re3}
\end{equation}
\end{lemma}

\begin{proof}
The conditions (\ref{re3}) are immediately verified by (\ref{gam}).
Moreover, by taking the Fourier transform of the r.h.s. of (\ref{re2}), we
get%
\begin{eqnarray*}
\mathcal{F}\left\{ \frac{\partial }{\partial t}f_{\Gamma }(x,t);\theta
\right\} &=&\frac{\partial }{\partial t}\left( 1-\frac{i\theta }{b}\right)
^{-t} \\
&=&-\left( 1-\frac{i\theta }{b}\right) ^{-t}\log \left( 1-\frac{i\theta }{b}%
\right) \\
&=&-\widehat{f}_{\Gamma }(\theta ,t)\log \left( 1-\frac{i\theta }{b}\right)
\\
&=&-\mathcal{F}\left\{ \mathcal{A}_{b,x}f_{\Gamma }(x,t);\theta \right\} .
\end{eqnarray*}
\end{proof}

From the previous Lemma we can conclude that the infinitesimal generator of
the Gamma process can be written as $\mathcal{A}_{x}=-\log (1+D_{x}).$

\section{Main Results}

\subsection{Univariate GS process}

By resorting to the representation (\ref{sub}) and applying the previous
results, we can obtain the differential equation satisfied by the density of
the univariate GS process $\mathcal{G}_{\alpha }^{\beta }$ . This can be
done, for $t>1$, by considering Lemma 1 together with the result (\ref{rs5})
on $\mathcal{S}_{\alpha }^{\beta }$, as follows: by (\ref{sub}), we can write%
\begin{equation}
g_{\alpha }^{\beta }(x,t)=\int_{0}^{\infty }p_{\alpha }^{\beta
}(x,z)f_{\Gamma }(z,t)dz.  \label{rs8}
\end{equation}%
We consider hereafter the simple case $b=1.$ We then apply (\ref{gam2}), for
$b=1$, and we get%
\begin{eqnarray*}
&&(1-e^{-\partial _{t}})g_{\alpha }^{\beta }(x,t) \\
&=&\int_{0}^{\infty }p_{\alpha }^{\beta }(x,z)(1-e^{-\partial
_{t}})f_{\Gamma }(z,t)dz \\
&=&\int_{0}^{\infty }p_{\alpha }^{\beta }(x,z)(1-e^{-\partial
_{t}})f_{\Gamma }(z,t)dz \\
&=&-\int_{0}^{\infty }p_{\alpha }^{\beta }(x,z)\frac{\partial }{\partial z}%
f_{\Gamma }(z,t)dz \\
&=&-[p_{\alpha }^{\beta }(x,z)f_{\Gamma }(z,t)]_{z=0}^{\infty
}+\int_{0}^{\infty }\frac{\partial }{\partial z}p_{\alpha }^{\beta
}(x,z)f_{\Gamma }(z,t)dz \\
&=&c\,^{RF}\mathcal{D}_{x,\beta }^{\alpha }\int_{0}^{\infty }p_{\alpha
}^{\beta }(x,z)f_{\Gamma }(z,t)dz=c\,^{RF}\mathcal{D}_{x,\beta }^{\alpha
}g_{\alpha }^{\beta }(x,t).
\end{eqnarray*}%
In the last step we have applied the first equation in (\ref{rs5}) and we
have considered that, for $t>1$, $f_{\Gamma }(0,t)=1.$ In the next theorem
we prove the same result in an alternative way, which can be applied for any
$t\geq 0.$

\begin{proposition}
The density $g_{\alpha }^{\beta }$ of the GS process $\mathcal{G}_{\alpha
}^{\beta }$ satisfies the following equation, for any $x,t\geq 0$ and $%
\alpha \in (0,2],$%
\begin{equation}
^{RF}\mathcal{D}_{x,\beta }^{\alpha }g_{\alpha }^{\beta }(x,t)=\frac{1}{c}%
(1-e^{-\partial _{t}})g_{\alpha }^{\beta }(x,t),  \label{eq}
\end{equation}%
with conditions%
\begin{equation}
\left\{
\begin{array}{l}
g_{\alpha }^{\beta }(x,0)=\delta (x) \\
\lim_{|x|\rightarrow \infty }g_{\alpha }^{\beta }(x,t)=0%
\end{array}%
\right. ,  \label{eqc}
\end{equation}%
where $c>0$ is the spreading rate of dispersion defined in (\ref{can2}).
\end{proposition}

\begin{proof}
By (\ref{rs8}) and (\ref{can2}) we can write the characteristic function of $%
\mathcal{G}_{\alpha }^{\beta }$ as%
\begin{eqnarray}
\mathbb{E}e^{i\theta \mathcal{G}_{\alpha }^{\beta }(t)} &=&\frac{1}{\Gamma
(t)}\int_{0}^{\infty }\exp \{cz\psi _{\beta }^{\alpha }(\theta
)\}z^{t-1}e^{-z}dz  \label{rs7} \\
&=&\left( \frac{1}{1-c\psi _{\beta }^{\alpha }(\theta )}\right) ^{t},  \notag
\end{eqnarray}%
where $\psi _{\beta }^{\alpha }(\theta )$ is defined in (\ref{ai}); thus the
Fourier transform of the space-fractional differential equation (\ref{eq})
can be written as%
\begin{eqnarray}
&&\mathcal{F}\left\{ ^{RF}\mathcal{D}_{x,\beta }^{\alpha }g_{\alpha }^{\beta
}(x,t);\theta \right\}  \label{eq2} \\
&=&\left[ \text{by (\ref{fel})}\right] =\psi _{\beta }^{\alpha }(\theta )%
\mathcal{F}\left\{ g_{\alpha }^{\beta }(x,t);\theta \right\}  \notag \\
&=&\psi _{\beta }^{\alpha }(\theta )\left( \frac{1}{1-c\psi _{\beta
}^{\alpha }(\theta )}\right) ^{t}.  \notag
\end{eqnarray}%
On the other hand we get
\begin{eqnarray*}
\frac{1}{c}(1-e^{-\partial _{t}})\mathcal{F}\left\{ g_{\alpha }^{\beta
}(x,t);\theta \right\} &=&\frac{1}{c}(1-e^{-\partial _{t}})\left( \frac{1}{%
1-c\psi _{\beta }^{\alpha }(\theta )}\right) ^{t} \\
&=&\frac{1}{c}\left( \frac{1}{1-c\psi _{\beta }^{\alpha }(\theta )}\right)
^{t}-\frac{1}{c}\left( \frac{1}{1-c\psi _{\beta }^{\alpha }(\theta )}\right)
^{t-1},
\end{eqnarray*}%
which coincides with (\ref{eq2}). The conditions (\ref{eqc}) are clearly
satisfied since
\begin{equation*}
g_{\alpha }^{\beta }(x,0)=\frac{1}{2\pi }\int_{-\infty }^{+\infty
}e^{-i\theta x}\left. \left( \frac{1}{1-c\psi _{\beta }^{\alpha }(\theta )}%
\right) ^{t}\right\vert _{t=0}d\theta =\delta (x)
\end{equation*}%
and $\lim_{|x|\rightarrow \infty }g_{\alpha }^{\beta }(x,t)=0$ (by (\ref{rs5}%
) and (\ref{rs8})).
\end{proof}

\subsubsection{Symmetric GS process}

In the special case of a symmetric GS process $\mathcal{G}_{\alpha }$ we can
easily derive from Proposition 3 the following result, which is expressed in
terms of the Riesz derivative $^{R}\mathcal{D}_{x}^{\alpha }$, defined in (%
\ref{rs2}). In its regularized form, for $\alpha \in (0,2],$ the derivative $%
^{R}\mathcal{D}_{x}^{\alpha }$ can be explicitly represented as%
\begin{equation}
^{R}\mathcal{D}_{x}^{\alpha }u(x)=\frac{\Gamma (1+\alpha )\sin (\pi \alpha
/2)}{\pi }\int_{0}^{\infty }\frac{u(x+y)-2u(x)+u(x-y)}{y^{1+\alpha }}dy,
\label{rie}
\end{equation}%
(see \cite{MAI}).

\begin{corollary}
The density $g_{\alpha }$ of the symmetric GS process $\mathcal{G}_{\alpha }$
satisfies the following equation, for any $x,t\geq 0$ and $\alpha \in (0,2],$%
\begin{equation}
^{R}\mathcal{D}_{x}^{\alpha }g_{\alpha }(x,t)=\frac{1}{c}(1-e^{-\partial
_{t}})g_{\alpha }(x,t),  \label{co}
\end{equation}%
where $c=\sigma ^{\alpha }$ and with conditions%
\begin{equation}
\left\{
\begin{array}{l}
g_{\alpha }(x,0)=\delta (x) \\
\lim_{|x|\rightarrow \infty }g_{\alpha }(x,t)=0%
\end{array}%
\right. .  \label{coc}
\end{equation}
\end{corollary}

\begin{remark}
We consider now some interesting special cases of the previous results. For $%
\alpha =1,$ we show, from the previous corollary, that the density $%
g_{1}(x,t)$ of a Cauchy process $\mathcal{C}$ subordinated to an independent
Gamma subordinator (i.e. the process defined as $\left\{ \mathcal{C}(\Gamma
(t)),t\geq 0\right\} $) satisfies the following equation, for any $x,t\geq 0$%
:%
\begin{equation*}
\frac{\partial }{\partial |x|}g_{1}(x,t)=\frac{1}{c}(1-e^{-\partial
_{t}})g_{1}(x,t),
\end{equation*}%
with conditions (\ref{coc}) and $\partial /\partial |x|:=\,^{R}\mathcal{D}%
_{x}^{1}$. For $\alpha =2,$ we derive the governing equation of the density $%
g_{2}(x,t)$ of the Variance Gamma process, since the latter can be
represented as a standard Brownian motion $B$ subordinated to an independent
Gamma subordinator, i.e. as $\left\{ B(\Gamma (t)),t\geq 0\right\} .$ Indeed
we get that $g_{2}(x,t)$\ satisfies, for any $x,t\geq 0,$ the second order
differential equation
\begin{equation*}
\frac{\partial ^{2}}{\partial x^{2}}g_{2}(x,t)=\frac{1}{c}(1-e^{-\partial
_{t}})g_{2}(x,t),
\end{equation*}%
where $c=\sigma ^{2}$ and with conditions (\ref{coc})$.$
\end{remark}

We derive now another equation satisfied by the density of the symmetric GS
process, which, unlike (\ref{co}), involves a standard time derivative and a
space fractional differential operator which generalizes (\ref{tem2}). Let
us define the fractional version of $\mathcal{A}_{k,x}$, for any $\alpha >0$%
, as%
\begin{equation}
\mathcal{A}_{k,x}^{\alpha }f(x):=\sum_{l=1}^{\infty }\frac{(-1/k)^{l+1}}{l}%
\,^{R}\mathcal{D}_{x}^{\alpha l}f(x),\quad x\geq 0,\text{ }k\in \mathbb{R},
\label{tema}
\end{equation}%
where $^{R}\mathcal{D}_{x}^{\nu }$ is the Riesz derivative of order $\nu >0$%
. We note that in the non-symmetric case (i.e. for $\beta \neq 0$) we can
not define the analogue to (\ref{tema}) since the Riesz-Feller derivative is
not defined for a fractional order greater than $2.$

\begin{proposition}
The density $g_{\alpha }$ of the symmetric GS process $\mathcal{G}_{\alpha }$
satisfies the following equation, for any $x,t\geq 0$ and $\alpha \in (0,2],$%
\begin{equation}
\frac{\partial }{\partial t}g_{\alpha }(x,t)=\mathcal{A}_{1/c,x}^{\alpha
}g_{\alpha }(x,t),  \label{gia}
\end{equation}%
where $c=\sigma ^{\alpha }$ and with conditions%
\begin{equation}
\left\{
\begin{array}{l}
g_{\alpha }(x,0)=\delta (x) \\
\lim_{|x|\rightarrow \infty }\frac{\partial ^{l}}{\partial x^{l}}g_{\alpha
}(x,t)=0,\qquad l=0,1,...%
\end{array}%
\right. .  \label{giac}
\end{equation}
\end{proposition}

\begin{proof}
The Fourier transform of (\ref{tema}) is given by%
\begin{eqnarray}
\mathcal{F}\left\{ \mathcal{A}_{1/c,x}^{\alpha }f(x);\theta \right\}
&=&\sum_{l=1}^{\infty }\frac{(-c)^{j+1}}{l}\int_{-\infty }^{+\infty
}e^{i\theta x}\,^{R}\mathcal{D}_{x}^{\alpha l}f(x)dx  \label{tema2} \\
&=&-\sum_{l=1}^{\infty }\frac{(-c)^{l+1}}{l}|\theta |^{\alpha l}\mathcal{F}%
\left\{ f(x);\theta \right\}  \notag \\
&=&-\log \left( 1+c|\theta |^{\alpha }\right) \mathcal{F}\left\{ f(x);\theta
\right\} .  \notag
\end{eqnarray}%
Therefore we get%
\begin{eqnarray}
\mathcal{F}\left\{ \mathcal{A}_{1/c,x}^{\alpha }g_{\alpha }(x,t);\theta
\right\} &=&\log \left( \frac{1}{1+c|\theta |^{\alpha }}\right) \mathcal{F}%
\left\{ g_{\alpha }(x,t);\theta \right\}  \label{exp} \\
&=&\log \left( \frac{1}{1+c|\theta |^{\alpha }}\right) \left( \frac{1}{%
1+c|\theta |^{\alpha }}\right) ^{t},  \notag
\end{eqnarray}%
since for $\beta =0$, the characteristic function (\ref{rs7}) reduces to%
\begin{equation}
\mathbb{E}e^{i\theta \mathcal{G}_{\alpha }(t)}=\left( \frac{1}{1+c|\theta
|^{\alpha }}\right) ^{t}.  \notag
\end{equation}%
The expression (\ref{exp}) clearly coincides with the Fourier transform of
the left-hand side of (\ref{gia}).
\end{proof}

The previous result agrees with the expression of the infinitesimal
generator $\mathcal{A}_{x}$ of the GS process, which is given by $\mathcal{A}%
_{x}=-\log \left[ 1+\left( -\frac{d^{2}}{dx^{2}}\right) ^{\alpha /2}\right] $
(see \cite{GRZ}).

\subsubsection{GS subordinator}

In the positively asymmetric case, i.e. for $\beta =1$, the process $%
\mathcal{G}_{\alpha }^{\beta }$ reduces to a GS subordinator (we will
denote it as $\mathcal{G}_{\alpha }^{\prime }$).

\begin{corollary}
The density $g_{\alpha }^{\prime }(x,t)$ of the GS subordinator $\mathcal{G}%
_{\alpha }^{\prime }$ satisfies the following equation, for any $x,t\geq 0$
and $\alpha \in (0,2],$%
\begin{equation}
^{RF}\mathcal{D}_{x,1}^{\alpha }g_{\alpha }^{\prime }(x,t)=\frac{1}{c}%
(1-e^{-\partial _{t}})g_{\alpha }^{\prime }(x,t),  \label{sub7}
\end{equation}%
where $c=\sigma ^{\alpha }\left( \cos (\pi \alpha /2)\right) ^{-1}$ and with
conditions%
\begin{equation}
\left\{
\begin{array}{l}
g_{\alpha }^{\prime }(x,0)=\delta (x) \\
\lim_{|x|\rightarrow \infty }g_{\alpha }^{\prime }(x,t)=0%
\end{array}%
\right. .  \label{sub8}
\end{equation}%
and $^{RF}\mathcal{D}_{x,1}^{\alpha }$ is the Riesz-Feller derivative
defined by $\mathcal{F}\left\{ ^{RF}\mathcal{D}_{x,1}^{\alpha }u(x);\theta
\right\} =-(-i|\theta |)^{\alpha }sign(\theta )\mathcal{F}\left\{
u(x);\theta \right\} .$
\end{corollary}

\begin{remark}
We now consider the special case $\alpha =1/2$ of the previous result. It is
well-known that the stable law with parameters $\alpha =1/2,$ $\mu =0,$ $%
\beta =1,$ $\sigma >0$\ coincides with the L\'{e}vy density. Moreover if we
define as%
\begin{equation*}
\mathcal{T}_{z}:=\inf_{s>0}\left\{ B(s)=z\right\} ,\qquad z\geq 0,
\end{equation*}%
the first-passage time of a standard Brownian motion $B,$ we have that%
\begin{equation*}
P\left\{ \mathcal{T}_{z}\in dx\right\} =p_{1/2}^{\prime }(x,z),\qquad
x,z\geq 0,
\end{equation*}%
since $\mathcal{T}_{z}$ is equal in distribution to a stable subordinator $%
\mathcal{S}_{1/2}^{\prime }$ of index $1/2$ and variance $\sigma =z^{2}$
(whose density is denoted as $p_{1/2}^{\prime }(x,z)$). Therefore, from the
previous corollary, we can derive that the density of the time-changed
process $\left\{ \mathcal{T}_{\Gamma (t)},t\geq 0\right\} ,$ given by%
\begin{equation*}
g_{1/2}^{\prime }(x,t)=\int_{0}^{\infty }p_{1/2}^{\prime }(x,z)f_{\Gamma
}(z,t)dz
\end{equation*}%
satisfies the following equation for any $x,t\geq 0$:%
\begin{equation}
\frac{\partial ^{1/2}}{\partial |x|^{1/2}}g_{1/2}^{\prime }(x,t)=\sqrt{2}%
(1-e^{-\partial _{t}})g_{1/2}^{\prime }(x,t),  \label{t}
\end{equation}%
with conditions (\ref{sub8}) and $\partial ^{1/2}/\partial |x|^{1/2}:=\,^{RF}%
\mathcal{D}_{x,1}^{1/2}.$ The constant in (\ref{t}) can be derived by
considering that, in this case, $c=\sqrt{\sigma }\left( \cos (\pi /4)\right)
^{-1}$ and we assume that $\sigma =1$. The process $\mathcal{T}_{\Gamma (t)}$
can be interpreted as the first-passage time of a Brownian motion through a
random barrier, represented by a Gamma process. Thus we can conclude that
\begin{equation*}
P\left\{ \inf_{s>0}\left\{ B(s)=\Gamma (t)\right\} \in dx\right\} ,\qquad
x,t\geq 0
\end{equation*}%
satisfies the space-fractional equation (\ref{t}).
\end{remark}

\subsection{Multivariate GS process}

The multivariate GS distribution was first defined in \cite{MIT} and applied
later to model multivariate financial portfolios of securities, in \cite{KOZ}%
.

In the $n$-dimensional case, we denote by $\left\{ \mathbf{G}_{\alpha
}^{n}(t),t\geq 0\right\} $ a multivariate GS process with stability index $%
\alpha \in (0,2]$, position parameter $\mathbf{\mu }=\mathbf{0}$ (for
simplicity) and spectral measure $M$, then its characteristic function can
be written as%
\begin{equation}
\mathbb{E}e^{i<\mathbf{\theta ,G}_{\alpha }^{n}(t)>}=\left[ 1+\int_{S^{n}}|<%
\mathbf{\theta ,z>}|^{\alpha }\omega _{\alpha ,1}(<\mathbf{\theta ,z>})M(d%
\mathbf{z})\right] ^{-t},\quad \mathbf{\theta }\in \mathbb{R}^{n},
\label{one}
\end{equation}%
where $S^{n}:=\{\mathbf{s}\in \mathbb{R}^{n}:||\mathbf{s}||=1\}$, $<\mathbf{%
\theta ,z>=}\sum_{j=1}^{n}\theta _{j}z_{j}$ and%
\begin{equation*}
\omega _{\alpha ,1}(<\mathbf{\theta ,z>}):=\left\{
\begin{array}{c}
1-isign(<\mathbf{\theta ,z>})\tan (\pi \alpha /2),\quad \text{if }\alpha
\neq 1 \\
1+2isign(<\mathbf{\theta ,z>})\log |\mathbf{\theta }|/\pi ,\quad \text{if }%
\alpha =1%
\end{array}%
\right. .
\end{equation*}%
Moreover, as in the univariate case, the following relationship holds for
the r.v. $\mathbf{G}_{\alpha }^{n}:=\mathbf{G}_{\alpha }^{n}(1)$:%
\begin{equation*}
\mathbb{E}e^{i\mathbf{\theta G}_{\alpha }^{n}}=\frac{1}{1-\log \Phi _{%
\mathbf{S}_{\alpha }^{n}}(\mathbf{\theta })},\quad \mathbf{\theta }\in
\mathbb{R}^{n}
\end{equation*}%
(see \cite{KOZ}), where
\begin{equation*}
\Phi _{\mathbf{S}_{\alpha }^{n}}(\theta ):=\mathbb{E}e^{i<\mathbf{\theta ,S}%
_{\alpha }^{n}>}=\exp \{-\int_{S^{n}}|<\mathbf{\theta ,z>}|^{\alpha }\omega
_{\alpha ,1}(<\mathbf{\theta ,z>})M(d\mathbf{z})\},\quad \mathbf{\theta }\in
\mathbb{R}^{n}
\end{equation*}%
is the characteristic function of a stable multivariate r.v. $\mathbf{S}%
_{\alpha }^{n}$ with $\mathbf{\mu }=\mathbf{0}$ and spectral measure $M$\
(see e.g. \cite{SAMO}, p.65).

Let the process $\left\{ \mathbf{S}_{\alpha }^{n}(t),t\geq 0\right\} $ be
defined by its characteristic function, i.e.%
\begin{equation*}
\Phi _{\mathbf{S}_{\alpha }^{n}(t)}(\theta ):=\mathbb{E}e^{i<\mathbf{\theta
,S}_{\alpha }^{n}(t)>}=\exp \{-t\int_{S^{n}}|<\mathbf{\theta ,z>}|^{\alpha
}\omega _{\alpha ,1}(<\mathbf{\theta ,z>})M(d\mathbf{z})\},\quad \mathbf{%
\theta }\in \mathbb{R}^{n}.
\end{equation*}%
Then the density $p_{\alpha }^{n}(\mathbf{x},t)$ of $\mathbf{S}_{\alpha
}^{n} $ satisfies the initial value problem, for $\alpha \in (0,2],$ $\alpha
\neq 1,$%
\begin{equation}
\left\{
\begin{array}{l}
\nabla _{M}^{\alpha }p_{\alpha }^{n}(\mathbf{x},t)=\frac{1}{c}\frac{\partial
}{\partial t}p_{\alpha }^{n}(\mathbf{x},t) \\
p_{\alpha }^{n}(\mathbf{x},0)=\delta (\mathbf{x})%
\end{array}%
\right. ,\quad \mathbf{x}\in \mathbb{R}^{n},t\geq 0,
\end{equation}%
where $c=(\cos (\pi \alpha /2))^{-1}$ (see \cite{MEE}, being careful with
the signs, for the different definition of Fourier transform) and $\nabla
_{M}^{\alpha }$ is the fractional derivative operator defined in (\ref{mgs}).

The results of the previous section can be generalized to the $n$%
-dimensional case, as follows.

\begin{proposition}
The density $g_{\alpha }^{n}(\mathbf{x},t)$ of the $n$-dimensional GS
process $\mathbf{G}_{\alpha }^{n}$ satisfies the following Cauchy problem,
for $\alpha \in (0,2],$ $\alpha \neq 1,$%
\begin{equation}
\left\{
\begin{array}{l}
\nabla _{M}^{\alpha }g_{\alpha }^{n}(\mathbf{x},t)=\frac{1}{c}%
(1-e^{-\partial _{t}})g_{\alpha }(\mathbf{x},t) \\
g_{\alpha }^{n}(\mathbf{x},0)=\delta (\mathbf{x}) \\
\lim_{||\mathbf{x}||\rightarrow \infty }g_{\alpha }^{n}(\mathbf{x},t)=0%
\end{array}%
\right. ,\;\mathbf{x\in }\mathbb{R}^{n},t\geq 0.  \label{adv2}
\end{equation}
\end{proposition}

\begin{proof}
The Fourier transform of the space-fractional differential equation in (\ref%
{adv2}) can be written as%
\begin{eqnarray*}
&&\mathcal{F}\left\{ \nabla _{M}^{\alpha }g_{\alpha }^{n}(\mathbf{x},t);%
\mathbf{\theta }\right\} \\
&=&\left[ \text{by (\ref{mgs})}\right] =-\left[ \int_{S^{n}}(-i<\mathbf{%
\theta ,z>})^{\alpha }M(d\mathbf{z})\right] \mathcal{F}\left\{ g_{\alpha
}^{n}(\mathbf{x},t);\mathbf{\theta }\right\} \\
&=&-\cos (\pi \alpha /2)\left[ \int_{S^{n}}|<\mathbf{\theta ,z>}|^{\alpha
}\omega _{\alpha ,1}(<\mathbf{\theta ,z>})M(d\mathbf{z})\right] \left[
1+\int_{S^{n}}|<\mathbf{\theta ,z>}|^{\alpha }\omega _{\alpha ,1}(<\mathbf{%
\theta ,z>})M(d\mathbf{z})\right] ^{-t} \\
&=&[\text{by (\ref{one})}] \\
&=&\cos (\pi \alpha /2)(1-e^{-\partial _{t}})\mathcal{F}\left\{ g_{\alpha
}^{n}(\mathbf{x},t);\mathbf{\theta }\right\} ,
\end{eqnarray*}%
by considering that%
\begin{equation*}
(-i<\mathbf{\theta ,z>})^{\alpha }=|<\mathbf{\theta ,z>}|^{\alpha }\cos (\pi
\alpha /2)\omega _{\alpha ,1}(<\mathbf{\theta ,z>}).
\end{equation*}%
The first condition in (\ref{adv2}) is verified since the characteristic
function of $\mathbf{G}_{\alpha }^{n}$, given in (\ref{one}), reduces to $1$
for $t=0,$ while for the second one we must consider that
\begin{equation*}
g_{\alpha }^{n}(\mathbf{x},t)=\int_{0}^{\infty }p_{\alpha }^{n}(\mathbf{x}%
,z)f_{\Gamma }(z,t)dz
\end{equation*}%
and that $\lim_{|x|\rightarrow \infty }p_{\alpha }^{n}(\mathbf{x},z)=0.$
\end{proof}

\begin{remark}
If we consider the special case of an isotropic $n$-dimensional GS process $%
\left\{ \mathbf{G}_{\alpha }(t),t\geq 0\right\} $, the previous results can
be considerably simplified. Indeed in this case we can use the fractional
Laplace operator defined by%
\begin{equation}
\mathcal{F}\left\{ \left( -\Delta \right) ^{\alpha }u(\mathbf{x);\theta }%
\right\} =-||\mathbf{\theta |}|^{\alpha }\mathcal{F}\left\{ u(\mathbf{x});%
\mathbf{\theta }\right\} ,\qquad \mathbf{x,\theta }\in \mathbb{R}^{n}
\label{rs3}
\end{equation}%
(where $||\cdot ||$ denotes the Euclidean norm) or, by the Bochner
representation, as%
\begin{equation}
\left( -\Delta \right) ^{\alpha }\mathbf{=-}\frac{\sin (\pi \alpha )}{\pi }%
\int_{0}^{+\infty }z^{\alpha -1}(z-\Delta )^{-1}\Delta dz  \label{rs4}
\end{equation}%
(see \cite{BOC} and \cite{BAL}). Moreover the $n$-dimensional isotropic GS
process is defined through its characteristic function
\begin{equation}
\mathbb{E}e^{i<\mathbf{\theta ,G}_{\alpha }(t)>}=\left( \frac{1}{1+||\mathbf{%
\theta }||^{\alpha }}\right) ^{t},\quad \mathbf{\theta }\in \mathbb{R}^{n}
\end{equation}%
(see \cite{KOZ}) and its marginals coincide with the multivariate Linnik
distributions introduced in \cite{AND}.\ The process $\mathbf{G}_{\alpha }$
can be represented as
\begin{equation}
\mathbf{G}_{\alpha }(t):=\mathbf{S}_{\alpha }(\Gamma (t)),\qquad t\geq 0,
\label{mul3}
\end{equation}%
where $\left\{ \mathbf{S}_{\alpha }(t),t>0\right\} $ is an isotropic stable
vector, with characteristic function%
\begin{equation*}
\mathbb{E}e^{i<\mathbf{\theta ,S}_{\alpha }(t)>}=\exp \{-t||\theta
||^{\alpha }\},\quad \mathbf{\theta }\in \mathbb{R}^{n}.
\end{equation*}%
Then it is well-known that the density $p_{\alpha }(\mathbf{x},t)$ of $%
\mathbf{S}_{\alpha }$ satisfies the equation%
\begin{equation}
\left\{
\begin{array}{l}
\left( -\Delta \right) ^{\alpha }p_{\alpha }(\mathbf{x},t)=\frac{1}{c}\frac{%
\partial }{\partial t}p_{\alpha }(\mathbf{x},t) \\
p_{\alpha }(\mathbf{x},0)=\delta (x) \\
\lim_{||\mathbf{x}||\rightarrow \infty }p_{\alpha }(\mathbf{x},t)=0%
\end{array}%
\right. ,  \label{mul}
\end{equation}%
for $x\in \mathbb{R}^{n}$, $t\geq 0,$ $c=(\cos (\pi \alpha /2))^{-1}$ and $%
\alpha \in (0,2].$ Therefore by Preposition 3, we can conclude that the
density $g_{\alpha }(\mathbf{x},t)$ of $\mathbf{G}_{\alpha }$ satisfies the
following Cauchy problem, for any $\mathbf{x\in }\mathbb{R}^{n},$ $t\geq 0$:%
\begin{equation}
\left\{
\begin{array}{l}
(-\Delta )^{\alpha }g_{\alpha }(\mathbf{x},t)=\frac{1}{c}(1-e^{-\partial
_{t}})g_{\alpha }(\mathbf{x},t) \\
g_{\alpha }(\mathbf{x},0)=\delta (\mathbf{x}) \\
\lim_{||\mathbf{x}||\rightarrow \infty }g_{\alpha }(\mathbf{x},t)=0%
\end{array}%
\right. ,  \label{mul5}
\end{equation}%
where $(-\Delta )^{\alpha }$ is the fractional Laplace operator defined in (%
\ref{rs3}).
\end{remark}

\

\end{document}